\documentclass{amsart}%
\usepackage{amssymb}
\usepackage{amsfonts}
\usepackage{amsmath}
\usepackage{graphicx}%
\providecommand{\U}[1]{\protect\rule{.1in}{.1in}}
\newtheorem{theorem}{Theorem}
\theoremstyle{plain}

\newtheorem{corollary}{Corollary}

\newtheorem{definition}{Definition}
\newtheorem{example}{Example}

\newtheorem{lemma}{Lemma}

\newtheorem{remark}{Remark}

\numberwithin{equation}{section}
\begin{document}
\title[ ]{A new look at the Hardy-Littlewood-P\'{o}lya inequality of majorization}
\author{Constantin P. Niculescu}
\address{University of Craiova, Department of Mathematics, A.I. Cuza Street 13, Craiova
200585, ROMANIA}
\email{constantin.p.niculescu@gmail.com}
\thanks{Published in \emph{J. Math. Anal. Appl.} \textbf{501} (2021), Issue 2, paper
125211. DOI: 10.1016/j.jmaa.2021.125211}
\dedicatory{ }
\begin{abstract}
The Hardy-Littlewood-P\'{o}lya inequality of majorization is extended to the
framework of ordered Banach spaces. Several applications illustrating our main
results are also included.

\end{abstract}
\maketitle

\section{Introduction}

In their celebrated book on \emph{Inequalities}, G. H. Hardy, J. E. Littlewood
and G. P\'{o}lya \cite{HLP} have proved an important characterization of
convex functions in terms of a preorder of vectors in $\mathbb{R}^{N}$ called
\emph{majorization}. Given two vectors $\mathbf{x}$ and $\mathbf{y}$ in
$\mathbb{R}^{N}$, we say that $\mathbf{x}$ is \emph{weakly majorized} by
$\mathbf{y}$ $($denoted $\mathbf{x}\prec_{wHLP}\mathbf{y})$ if their
decreasing rearrangements, respectively $x_{1}^{\downarrow}\geq\cdots\geq
x_{N}^{\downarrow}$\ and $y_{1}^{\downarrow}\geq\cdots\geq y_{N}^{\downarrow}$
verify the inequalities
\begin{equation}
\sum_{i=1}^{k}x_{i}^{\downarrow}\leq\sum_{i=1}^{k}y_{i}^{\downarrow}%
\quad\text{for }k=1,\dots,N; \label{wmaj}%
\end{equation}
we say that $\mathbf{x}$ is \emph{majorized} by $\mathbf{y}$ $($denoted
$\mathbf{x}\prec_{HLP}\mathbf{y})$ if in addition
\begin{equation}
\sum_{i=1}^{N}x_{i}^{\downarrow}=\sum_{i=1}^{N}y_{i}^{\downarrow}
\label{eqsums}%
\end{equation}

The basic result relating majorization to convexity is the
\emph{Hardy-Littlewood-P\'{o}lya inequality} \emph{of} \emph{majorization}:

\begin{theorem}
\label{thmHLP}$($Hardy-Littlewood-P\'{o}lya \emph{\cite{HLP}}$)$ If
$\mathbf{x}\prec_{HLP}\mathbf{y},$ then%
\begin{equation}
\sum_{k=1}^{N}f(x_{k})\leq\sum_{k=1}^{N}f(y_{k}) \label{HLPIneq}%
\end{equation}
for every real-valued continuous convex function $f$ defined on an interval
that contains the components of $\mathbf{x}$ and~$\mathbf{y}\mathrm{.}%
\vspace{1mm}$ Conversely, if the inequality $($\emph{\ref{HLPIneq}}$)$ holds
for every real-valued continuous convex function defined on an interval
including the components of $\mathbf{x}$ and $\mathbf{y},$ then $\mathbf{x}%
\prec_{HLP}\mathbf{y}.$
\end{theorem}

The inequality (\ref{HLPIneq}) still works when $f$ is a nondecreasing convex
function and $\mathbf{x}\prec_{wHLP}\mathbf{y}.$ This important remark, due
independently to Tomi\'{c} and Weyl, can be derived directly from Theorem
\ref{thmHLP}. See \cite{MOA2011} and \cite{NP2018}.

As we briefly noticed in \cite{NO}, the Hardy-Littlewood-P\'{o}lya inequality
can be extended to the framework of ordered Banach spaces alongside an
argument that can be traced back to \cite{MPP1995}. The aim of the present
paper is to prove much more general results and to show that they are best
possible (that is, no such theorems exist with less restrictions than ours).

The necessary background on ordered Banach spaces can be covered from
\cite{NO}. Additional information is available in the classical books of
Aliprantis and Tourky \cite{AT2007} and Meyer-Nieberg \cite{MN}.

According to Choquet's theory (see \cite{NP2018} and \cite{Ph2001}), the right
framework for developing the majorization theory is that of probability
measures. In the case of the Hardy-Littlewood-P\'{o}lya preorder relation
$\prec_{HLP}~$this can be done simply by identifying each vector
$\mathbf{x}=(x_{1},...,x_{N})$ in $\mathbb{R}^{N}$ with the discrete
probability measure $\left(  1/N\right)  \sum_{k=1}^{N}\delta_{x_{k}}$ acting
on $\mathbb{R};$ as usually $\delta_{x_{k}}$ denotes the Dirac measure
concentrated at $x_{k}$. We put
\[
\frac{1}{N}\sum_{k=1}^{N}\delta_{x_{k}}\prec_{HLP}\frac{1}{N}\sum_{k=1}%
^{N}\delta_{y_{k}}%
\]
with the same understanding as $\mathbf{x}\prec_{HLP}\mathbf{y}.$ Under these
terms, the Hardy-Littlewood-P\'{o}lya inequality of majorization can be
rephrased as%
\[
\mu\prec_{HLP}\nu\text{ if and only if }\int_{I}f\mathrm{d}\mu\leq\int
_{I}f\mathrm{d}\nu
\]
for every real-valued continuous and convex function $f$ whose domain of
definition is an interval $I$ that includes the supports of the discrete
probability measures$\ \mu$ and~$\nu\mathrm{.}\vspace{1mm}$

Since in an ordered Banach space not every string of elements admits a
decreasing rearrangement, in this paper we will concentrate to the case of
pairs of discrete probability measures of which at least one of them is
supported by a monotone string of points. The case where the support of the
left measure consists of a decreasing string is defined as follows.

\begin{definition}
\label{def1}Suppose that $\sum_{k=1}^{N}\lambda_{k}\delta_{\mathbf{x}_{k}}$
and $\sum_{k=1}^{N}\lambda_{k}\delta_{\mathbf{y}_{k}}$ are two discrete Borel
probability measures that act on the ordered Banach space $E$. We say that
$\sum_{k=1}^{N}\lambda_{k}\delta_{\mathbf{x}_{k}}$ is weakly $L^{\downarrow}%
$-majorized by $\sum_{k=1}^{N}\lambda_{k}\delta_{\mathbf{y}_{k}}$
\emph{(}denoted \newline$\sum_{k=1}^{N}\lambda_{k}\delta_{\mathbf{x}_{k}}%
\prec_{wL^{\downarrow}}\sum_{k=1}^{N}\lambda_{k}\delta_{\mathbf{y}_{k}}%
$\emph{)} if the left hand measure is supported by a decreasing string of
points
\begin{equation}
\mathbf{x}_{1}\geq\cdots\geq\mathbf{x}_{N} \label{x}%
\end{equation}
and%
\begin{equation}
\sum_{k=1}^{n}\lambda_{k}\mathbf{x}_{k}\leq\sum_{k=1}^{n}\lambda_{k}%
\mathbf{y}_{k}\quad\text{for all }n\in\{1,\dots,N\}. \label{maj1}%
\end{equation}

We say that $\sum_{k=1}^{N}\lambda_{k}\delta_{\mathbf{x}_{k}}$ is
$L^{\downarrow}$-majorized by $\sum_{k=1}^{N}\lambda_{k}\delta_{\mathbf{y}%
_{k}}$ $($denoted \linebreak$\sum_{k=1}^{N}\lambda_{k}\delta_{\mathbf{x}_{k}%
}\prec_{L^{\downarrow}}\sum_{k=1}^{N}\lambda_{k}\delta_{\mathbf{y}_{k}})$ if
in addition%
\begin{equation}
\sum_{k=1}^{N}\lambda_{k}\mathbf{x}_{k}=\sum_{k=1}^{N}\lambda_{k}%
\mathbf{y}_{k}. \label{maj2}%
\end{equation}

\end{definition}

Notice that the context of Definition \ref{def1} makes necessary that all
weights $\lambda_{1},...,\lambda_{N}$ belong to $(0,1]$ and $\sum_{k=1}%
^{N}\lambda_{k}=1.$

The three conditions (\ref{x}), (\ref{maj1}) and (\ref{maj2}) imply
$\mathbf{y}_{N}\leq\mathbf{x}_{N}\leq\mathbf{x}_{1}\leq$ $\mathbf{y}_{1}$ but
not the ordering $\mathbf{y}_{1}\geq\cdots\geq\mathbf{y}_{N}.$ For example,
when $N=3,$ one may choose
\[
\lambda_{1}=\lambda_{2}=\lambda_{3}=1/3,\text{ }\mathbf{x}_{1}=\mathbf{x}%
_{2}=\mathbf{x}_{3}=\mathbf{x}%
\]
and%
\[
\mathbf{y}_{1}=\mathbf{x},\text{ }\mathbf{y}_{2}=\mathbf{x+z,}\text{
}\mathbf{y}_{3}=\mathbf{x-z}%
\]
where $\mathbf{z}$ is any positive element.

Under these circumstances it is natural to introduce the following companion
to Definition 1, involving the ascending strings of elements as support for
the right hand measure.

\begin{definition}
\label{def2}The relation of\emph{ weak }$R^{\uparrow}$-\emph{majorization},
\[
\sum_{k=1}^{N}\lambda_{k}\delta_{\mathbf{x}_{k}}\prec_{wR^{\uparrow}}%
\sum_{k=1}^{N}\lambda_{k}\delta_{\mathbf{y}_{k}},
\]
between two discrete Borel probability measures means the fulfillment of the
condition $($\emph{\ref{maj1}}$)$ under the presence of the ordering%
\begin{equation}
\mathbf{y}_{1}\leq\cdots\leq\mathbf{y}_{N}; \label{y}%
\end{equation}
assuming in addition the condition $($\emph{\ref{maj2}}$)$\emph{,} we say that
$\sum_{k=1}^{N}\lambda_{k}\delta_{\mathbf{x}_{k}}$ is $R^{\uparrow}$-majorized
by $\sum_{k=1}^{N}\lambda_{k}\delta_{\mathbf{y}_{k}}$ $($denoted $\sum
_{k=1}^{N}\lambda_{k}\delta_{\mathbf{x}_{k}}\prec_{R^{\uparrow}}\sum_{k=1}%
^{N}\lambda_{k}\delta_{\mathbf{y}_{k}})$.
\end{definition}

When every element of $E$ is the difference of two positive elements, the weak
majorization relations $\prec_{wL^{\downarrow}}$and $\prec_{wR^{\uparrow}}$
can be augmented so to obtain majorization relations.

The corresponding extensions of the Hardy-Littlewood-P\'{o}lya inequality of
majorization for $\prec_{wL^{\downarrow}},\prec_{L^{\downarrow}}%
,\prec_{wR^{\uparrow}}$and $\prec_{R^{\uparrow}}$make the objective of two
theorems in Section 4. The first one, Theorem \ref{thmHLPgen}, deals with
G\^{a}teaux differentiable convex functions whose differentials are isotone
(that is, order preserving). The second one, Theorem \ref{thmHLPgen2}, extends
the conclusion of the preceding theorem to a nondifferentiable framework
involving convex functions defined on open $N$-dimensional box of
$\mathbb{R}^{N}$ which verify a condition of monotonicity \`{a} la Popoviciu
\cite{Pop34} (called by us $2$-box monotonicity). This is done via the
approximation Theorem \ref{thmappr}, whose proof makes the objective of
Section 3.

Unlike the case of functions of one real variable, when the isotonicity of the
differential is automatic, for several variables, while this is not
necessarily true in the case of a differentiable convex function of a vector
variable. See Remark \ref{rem4}. Remarkably, the isotonicity of the
differential is not only a sufficient condition for the validity of our
differentiable generalization of the Hardy-Littlewood-P\'{o}lya theorem, but
also a necessary one. See Remark \ref{remsuf}.

For the convenience of the reader we review in Section 2 some very basic
results concerning the various classes of convex or convex like functions and
the gradient inequalities they generate. This section also includes several
significant examples of differentiable convex functions with isotone differentials.

Not entirely surprising, the inequalities of majorization may occur outside
the class of convex functions. This is illustrated by Theorem \ref{thmHLPgen2}%
, that deals with the case of strongly smooth functions.

Applications of the above majorization theorems include the isotonicity of
Jensen's gap, a general form of the parallelogram law and also the extension
of several classical inequalities to the setting of convex functions of a
vector variable. They are all presented in Section 5.

\section{Classes of convex functions}

In what follows $E$ and $F$ are two ordered Banach spaces and $\Phi
:C\rightarrow F$ is a function defined on a convex subset of $E.$

The function $\Phi$ is said to be a \emph{perturbed} \emph{convex function}
\emph{with modulus} $\omega:\mathbb{[}0,\infty\mathbb{)}\rightarrow F$
(abbreviated, $\omega$-\emph{convex function}) if it verifies an estimate of
the form
\begin{equation}
\Phi((1-\lambda)\mathbf{x}+\lambda\mathbf{y})\leq(1-\lambda)\Phi
(\mathbf{x})+\lambda\Phi(\mathbf{y})-\lambda(1-\lambda)\omega\left(
\left\Vert \mathbf{x}-\mathbf{y}\right\Vert \right)  ,\text{\quad}
\label{omega_conv}%
\end{equation}
for all $\mathbf{x},\mathbf{y}$ in $C$ and $\lambda\in(0,1).$ The usual convex
functions represent the particular case when $\omega$ is identically 0. Every
$\omega$-convex function associated to a modulus\emph{ }$\omega\geq0$ is
necessarily convex. When $\omega$ is allowed to take negative values, there
are $\omega$-convex functions which are not convex. See the case of semiconvex
functions, described below.

The $\omega$-convex functions whose moduli $\omega$ are strictly positive
except at the origin (where $\omega(0)=0)$ are usually called \emph{uniformly
convex}.\emph{ }In their case the inequality (\ref{omega_conv}) is strict
whenever $\mathbf{x}\neq\mathbf{y}$ and $\lambda\in(0,1).$\emph{ }More
information on this important class of convex functions is available in
\cite{AP}, \cite{BGHV}, \cite{BV}, \cite{Z1983} and \cite{Zal2002}.

By changing $\Phi$ to $-\Phi$ one obtain the notions of $\omega$-\emph{concave
function }and\emph{ uniformly concave function. }

\begin{remark}
\label{rem1}Much of the study of perturbed convex functions with vector values
can be reduced to that of real-valued functions.\ Indeed, in any ordered
Banach space $F$ with generating cone, any inequality of the form$\ u\leq v$
is equivalent to $w^{\ast}(u)\leq w^{\ast}(v)$ for all $w^{\ast}\in
F_{+}^{\ast}.$ See \emph{\cite{NO}}, Lemma $1$ $(c)$. As a consequence, a
function $\Phi:C\rightarrow F$ is $\omega$-convex if and only if $w^{\ast
}\circ\Phi$ is $(w^{\ast}\circ\omega)$-convex whenever $w^{\ast}\in
F_{+}^{\ast}.$
\end{remark}

There are several variants of convexity that play a prominent role in convex
optimization, calculus of variations, isoperimetric inequalities,
Monge--Kantorovich theory of transport etc. Some of them are mentioned in what follows.

A real-valued function $\Phi$ defined on a convex subset $C$ of $\mathbb{R}%
^{N}$ is called $\alpha$-\emph{strongly convex} functions (that is, strongly
convex with parameter $\alpha>0)$ if $\Phi-$ $(\alpha/2)\left\Vert
\cdot\right\Vert ^{2}$ is convex. The function $\Phi$ is called $\beta
$-\emph{semiconvex} (that is, semiconvex with parameter $\beta>0)$ if it
becomes convex after the addition of $(\beta/2)\left\Vert \cdot\right\Vert
^{2}.$ Equivalently, these are the functions that verify respectively
estimates of the form%
\begin{equation}
\Phi((1-\lambda)\mathbf{x}+\lambda\mathbf{y})\leq(1-\lambda)\Phi
(\mathbf{x})+\lambda\Phi(\mathbf{y})-\frac{1}{2}\lambda(1-\lambda
)\alpha\left\Vert \mathbf{x}-\mathbf{y}\right\Vert ^{2}
\label{M_strongly conv}%
\end{equation}
and
\begin{equation}
\Phi((1-\lambda)\mathbf{x}+\lambda\mathbf{y})\leq(1-\lambda)\Phi
(\mathbf{x})+\lambda\Phi(\mathbf{y})+\frac{1}{2}\lambda(1-\lambda
)\beta\left\Vert \mathbf{x}-\mathbf{y}\right\Vert ^{2},\text{\quad}
\label{L_semiconv}%
\end{equation}
for all $\mathbf{x},\mathbf{y}$ in $C$ and $\lambda\in(0,1)$. By changing
$\Phi$ to $-\Phi$ one obtain the notions of $\alpha$-\emph{strong concavity}
and $\beta$-\emph{semiconcavity}.

Under the presence of G\^{a}teaux differentiability, each of the above classes
of functions generates specific gradient inequalities that play a prominent
role in our generalization of the Hardy-Littlewood-P\'{o}lya inequality of majorization.

\begin{lemma}
\label{lem2}Suppose that $C$ is an open convex subset of $E$ and
$\Phi:C\rightarrow F$ is a function both G\^{a}teaux differentiable and
$\omega$-convex. Then\emph{ }
\begin{equation}
\Phi(\mathbf{x})-\Phi(\mathbf{a})\geq\Phi^{\prime}(\mathbf{a})(\mathbf{x}%
-\mathbf{a})+\omega\left(  \left\Vert \mathbf{x}-\mathbf{a}\right\Vert
\right)  \label{grad_ineq}%
\end{equation}
for all points $\mathbf{a}\in C$ and $\mathbf{x}\in C$.
\end{lemma}

As is well known, if $C$ is an open convex subset of $\mathbb{R}^{N},$ then a
twice continuously differentiable function $\Phi:C\rightarrow\mathbb{R}$ is
$\alpha$-strongly convex (respectively $\beta$-semiconvex) if and only if its
Hessian matrix verifies the inequality $\nabla^{2}\Phi\geq\alpha I$
(respectively $\nabla^{2}\Phi\geq-\beta I).$ However, valuable
characterizations are possible with less smoothness.

A continuously differentiable function $\Phi:C\rightarrow\mathbb{R}$ defined
on an open convex subset of $\mathbb{R}^{N}$ is said to be $\sigma
$-\emph{strongly} \emph{smooth} if its gradient is $\sigma$-Lipschitz, that
is,%
\[
\left\Vert \nabla\Phi(\mathbf{x})-\nabla\Phi(\mathbf{y})\right\Vert \leq
\sigma\left\Vert \mathbf{x}-\mathbf{y}\right\Vert \text{\quad for all
}\mathbf{x},\mathbf{y}\in C.
\]

Notice that every $\sigma$-strongly smooth function $\Phi$ verifies the
following variant of the gradient inequality:
\begin{equation}
\Phi(\mathbf{y})-\Phi(\mathbf{x})\leq\Phi^{\prime}(\mathbf{x})(\mathbf{y}%
-\mathbf{x})+\frac{1}{2}\sigma\left\Vert \mathbf{y}-\mathbf{x}\right\Vert ^{2}
\label{smooth_grad}%
\end{equation}
for all $\mathbf{x},\mathbf{y}$ in $C.$ See \cite{Bub}, Lemma 3.4, p. 267.

\begin{lemma}
\label{lem3}If $\Phi$ is simultaneously convex and $\sigma$-strongly smooth,
then%
\[
\frac{1}{2\sigma}\left\Vert \Phi^{\prime}(\mathbf{y})-\Phi^{\prime}%
(\mathbf{x})\right\Vert \leq\Phi(\mathbf{y})-\Phi(\mathbf{x})-\Phi^{\prime
}(\mathbf{x})(\mathbf{y}-\mathbf{x})\leq\frac{1}{2}\sigma\left\Vert
\mathbf{y}-\mathbf{x}\right\Vert ^{2}.
\]

\end{lemma}

\begin{proof}
The left-hand side inequality is just Lemma 3.5 in \cite{Bub}, p. 268, while
the right-hand side inequality is a restatement of the inequality
(\ref{smooth_grad}).
\end{proof}

An important source of strongly smooth convex functions is offered by the
following result:

\begin{lemma}
\label{lem4}If $\Phi$ is an $\alpha$-strongly convex function, then its
Legendre-Fenchel conjugate%
\[
\Phi^{\ast}(\mathbf{x}^{\ast})=\sup\left\{  \mathbf{x}^{\ast}(\mathbf{x}%
)-\Phi(\mathbf{x}):\mathbf{x}\in C\right\}
\]
is an $(1/\alpha)$-strongly smooth function and also a convex function. In
particular, $\Phi^{\ast}$ is defined and differentiable on the whole dual
space $E^{\ast}.$
\end{lemma}

For details, see \cite{ShSh}, Lemma 15, p. 126. The converse also works. See
\cite{KST}, Theorem 6.

The connection between $\sigma$-strong smoothness and semiconvexity is
outlined by the following theorem.

\begin{theorem}
$(a)$ Suppose that $C$ is an open convex subset of $\mathbb{R}^{N}.$ If
$\Phi:C\rightarrow\mathbb{R}$ is a $\sigma$-strongly smooth function, then
$\Phi+\frac{\sigma}{2}\left\Vert \cdot\right\Vert ^{2}$ is convex and
$\Phi-\frac{\sigma}{2}\left\Vert \cdot\right\Vert ^{2}$ is concave.

$(b)$ Conversely, if $\Phi:C\rightarrow\mathbb{R}$ is a function
simultaneously semiconvex and semiconcave with parameter $\sigma>0,$ then
$\Phi$ is $\sigma$-strongly smooth.
\end{theorem}

The details are available in the book of Cannarsa and Sinestrari
\cite{CS2004}; the assertion $(a)$ follows from Proposition 2.1.2, p. 30,
while the assertion $(b)$ is motivated by Corollary 3.3.8, p. 61.

As was noticed by Amann \cite{Amann1974}, Proposition 3.2, p. 184, the
G\^{a}teaux differentiability offers a convenient way to recognize the
property of isotonicity of functions acting on ordered Banach spaces: the
positivity of the differential. In the case of convex functions his result can
be stated as follows:

\begin{lemma}
\label{lemAmann_conv}Suppose that $E$ and $F$ are two ordered Banach space,
$C$ is a convex subset of $E$ with nonempty interior $\operatorname{int}C$ and
$\Phi:C\rightarrow F$ is a \emph{ }convex function, continuous on $C$ and
G\^{a}teaux differentiable on $\operatorname{int}C.$ Then $\Phi$ is isotone on
$C$ if and only if $\Phi^{\prime}(\mathbf{a})\geq0$ for all $\mathbf{a}%
\in\operatorname{int}C.$
\end{lemma}

\begin{proof}
The "only if" part follows immediately from the definition of the G\^{a}teaux
derivative. For the other implication, notice that the gradient inequality
mentioned by Lemma $2$ shows that $\Phi$ is isotone on $\operatorname{int}C$
if $\Phi^{\prime}(\mathbf{a})\geq0$ for all $\mathbf{a}\in\operatorname{int}%
C$. As concerns the isotonicity on $C,$ that follows by an approximation
argument. Suppose that $\mathbf{x},\mathbf{y}\in C$ and $\mathbf{x}%
\leq\mathbf{y}$. For $\mathbf{x}_{0}\in\operatorname*{int}C$ arbitrarily fixed
and $t\in\lbrack0,1),$ both elements $\mathbf{u}_{t}=\mathbf{x}_{0}%
+t(\mathbf{x}-\mathbf{x}_{0})$ and $\mathbf{v}_{t}=\mathbf{x}_{0}%
+t(\mathbf{y}-\mathbf{x}_{0})$ belong to $\operatorname*{int}C$ and
$\mathbf{u}_{t}\leq\mathbf{v}_{t}.$ Moreover, $\mathbf{u}_{t}\rightarrow
\mathbf{x}$ and $\mathbf{v}_{t}\rightarrow\mathbf{y}$ as $t\rightarrow1.$
Passing to the limit in the inequality $\Phi(\mathbf{u}_{t})\leq
\Phi(\mathbf{v}_{t})$ we conclude that $\Phi(\mathbf{x})\leq\Phi(\mathbf{y}).$
\end{proof}

\begin{remark}
\label{rem3}If the ordered Banach space $E$ has finite dimension, then the
statement of Lemma \emph{\ref{lemAmann_conv}} remains valid by replacing the
interior of $C$ by the relative interior of $C$. See \emph{\cite{NP2018}},
Exercise $6$, p. $81$.
\end{remark}

A key ingredient in our extension of the Hardy-Littlewood-P\'{o}lya inequality
is the isotonicity of the differentials of the functions involved. Unlike the
case of differentiable convex functions of one variable, the isotonicity of
the differential is not mandatory for the differentiable convex functions of
several variables.

\begin{remark}
\label{rem4}$($A difference between the differentiable convex functions of one
real variable and those of several variables$)$ The twice continuously
differentiable function%
\[
\Phi(x,y)=-2\left(  xy\right)  ^{1/2},\quad\left(  x,y\right)  \in
\mathbb{R}_{++}^{2},
\]
is convex due to the fact that its Hessian,%
\[
H=\frac{1}{2}\left(
\begin{array}
[c]{cc}%
x^{-3/2}y^{1/2} & -x^{-1/2}y^{-1/2}\\
-x^{-1/2}y^{-1/2} & x^{1/2}y^{-3/2}%
\end{array}
\right)  ,
\]
is a positive semidefinite matrix. However, unlike the case of convex
functions of one real variable, the differential of $\Phi$,%
\[
d\Phi:\mathbb{R}_{++}^{2}\rightarrow\mathbb{R}^{2},\text{\quad}d\Phi
(x,y)=-(x^{-1/2}y^{1/2},x^{1/2}y^{-1/2}),
\]
is not isotone. Indeed, at the points $\left(  1,1\right)  <(2,1)$ in
$\mathbb{R}_{++}^{2}$ we have
\[
d\Phi(1,1)=-\left(  1,1\right)  \text{ and }d\Phi(2,1)=-(1/\sqrt{2},\sqrt{2})
\]
and these values are not comparable.

On the other hand, a simple example of nonconvex differentiable function whose
differential is isotone is provided by the function%
\[
H(x,y)=(2x-1)(2y-1),\quad\left(  x,y\right)  \in\mathbb{R}^{2}.
\]

\end{remark}

Using the aforementioned result of Amann, one can easily prove the following
criterion of isotonicity of the differentials.

\begin{lemma}
\label{lemconsamann}Suppose that $C$ is an open convex subset of the Banach
lattice $\mathbb{R}^{N}$ and $\Phi:C\rightarrow\mathbb{R}$ is a continuous
function which is twice G\^{a}teaux differentiable. Then $\Phi^{\prime}$ is
isotone on $C$ if $($and only if$)$ all partial derivatives of second order of
$\Phi$ are nonnegative.

When $\Phi$ is also convex, the isotonicity of $\Phi^{\prime}$ is equivalent
to the condition that all mixed derivatives $\frac{\partial^{2}\Phi}{\partial
x_{i}\partial x_{j}}$ are nonnegative.
\end{lemma}

In the light of Lemma \ref{lemconsamann}, the example exhibited in Remark
\ref{rem4}, shows that the property of positive definiteness of the Hessian
matrix does not necessarily imply its positivity as a linear map from
$\mathbb{R}^{2}$ to $\mathbb{R}^{2}$.

Several examples of differentiable functions which are isotone and/or admit
isotone differentials are presented in the Appendix of this paper.

\section{An approximation result}

One can characterize the isotonicity of the differential of a convex function
by using the concept of $2$-box monotonicity, first noticed by Popoviciu
\cite{Pop34} in the case when $N=2.$ See also \cite{GN2019}, where the $2$-box
monotonicity is described in its relationship with another concept due to
Popoviciu, $2$-box convexity. The natural domains of such functions are the
open $N$-dimensional boxes, that is, the products $\prod\nolimits_{k=1}%
^{N}(a_{k},b_{k})$ of $N$ open intervals.

\begin{definition}
\label{def2boxmon}A real-valued function $\Phi$ defined on an open and solid
subset $C$ of the Banach lattice $\mathbb{R}^{N}$ $(N\geq2)$ is $2$-box
monotone if the increment of $\Phi$ over every nondegenerate $2$-dimensional
box%
\[
B_{ij}=\{u_{1}\}\times\cdots\times\lbrack v_{i},w_{i}]\times\cdots
\times\left\{  u_{k}\right\}  \times\cdots\times\lbrack v_{j},w_{j}%
]\times\cdots\times\left\{  u_{N}\right\}  ,\quad1\leq i<j\leq N,
\]
included in $C$ and parallel to one of the planes of coordinates is
nonnegative, that is,
\begin{multline*}
\Delta(\Phi;B_{ij})=\Phi(u_{1},...,v_{i},...,v_{j},...,u_{N})-\Phi
(u_{1},...,v_{i},...,w_{j},...,u_{N})\\
-\Phi(u_{1},...,w_{i},...,v_{j},...,u_{N})+\Phi(u_{1},...,w_{i},...,w_{j}%
,...,u_{N})\geq0.
\end{multline*}

\end{definition}

The property of isotonicity of the differential of a convex function (of two
or more variables) is equivalent to the property of 2-box monotonicity for the
given function. When $\Phi$ is twice continuously differentiable, this follows
directly from Lemma \ref{lemconsamann}. Indeed, for $i<j,$%
\begin{multline*}
\int\nolimits_{v_{i}}^{w_{i}}\int\nolimits_{v_{j}}^{w_{j}}\frac{\partial
^{2}\Phi}{\partial x_{i}\partial x_{j}}(u_{1},...,x_{i},...,x_{j}%
,...,u_{N})\mathrm{d}x_{i}\mathrm{d}x_{j}\\
=\Phi(u_{1},...,v_{i},...,v_{j},...,u_{N})-\Phi(u_{1},...,v_{i},...,w_{j}%
,...,u_{N})\\
-\Phi(u_{1},...,w_{i},...,v_{j},...,u_{N})+\Phi(u_{1},...,w_{i},...,w_{j}%
,...,u_{N}).
\end{multline*}

Remarkably, the continuous differentiability of $\Phi$ suffices as well.

\begin{lemma}
\label{lem2box}Suppose that $C$ is an open box of the Banach lattice
$\mathbb{R}^{N}$ and $\Phi:C\rightarrow\mathbb{R}$ is a continuously
differentiable convex function. Then $\Phi^{\prime}$ is isotone on $C$ if
$($and only if$)$ $\Phi$ is $2$-box monotone.\ 
\end{lemma}

\begin{proof}
The fact that $\Phi^{\prime}$ is isotone is equivalent to%
\begin{equation}
\frac{\partial\Phi}{\partial x_{k}}(u_{1},...,u_{N})\leq\frac{\partial\Phi
}{\partial x_{k}}(v_{1},...,v_{N}) \label{ineqmonder}%
\end{equation}
for all indices $k\in\{1,...,N\}$ and all pairs of points $\mathbf{u}%
=(u_{1},...,u_{N})\leq\mathbf{v=}(v_{1},...,v_{N})$ in $C.$

Since $\Phi$ is differentiable and convex in each variable, we have%
\[
\frac{\partial\Phi}{\partial x_{k}}(u_{1},...,u_{k-1},x_{k},u_{k+1}%
,...,u_{N})\leq\frac{\partial\Phi}{\partial x_{k}}(u_{1},...,u_{k-1}%
,y_{k},u_{k+1},...,u_{N})
\]
whenever $(u_{1},...,u_{k-1},x_{k},u_{k+1},...,u_{N})\leq(u_{1},...,u_{k-1}%
,y_{k},u_{k+1},...,u_{N})$ in $C.$

Using the identity%
\begin{align*}
&  \int_{x_{j}}^{y_{j}}\left(  \frac{\partial\Phi}{\partial x_{j}}%
(u_{1},...,y_{i},...,t,...,u_{N})-\frac{\partial\Phi}{\partial x_{j}}%
(u_{1},...,x_{i},...,t,...,u_{N})\right)  \mathrm{d}t\\
&  =\Phi(x_{1},...,y_{i},...,y_{j},...,x_{N})-\Phi(x_{1},...,y_{i}%
,...,x_{j},...,x_{N})\\
&  -\Phi(x_{1},...,x_{i},...,y_{j},...,x_{N})+\Phi(x_{1},...,x_{i}%
,...,x_{j},...,x_{N}),
\end{align*}
which works for every nondegenerate $2$-dimensional box%
\[
\{u_{1}\}\times\cdots\times\lbrack x_{i},y_{i}]\times\cdots\times\left\{
u_{k}\right\}  \times\cdots\times\lbrack x_{j},y_{j}]\times\cdots
\times\left\{  u_{N}\right\}
\]
included in $C,$ we can easily infer that the $2$-box monotonicity is
equivalent with the isotonicity of each partial derivative $\frac{\partial
\Phi}{\partial x_{k}}$ in each variable distinct from the $k$th, when the
others are kept fixed. By mathematical induction one can put together all
these facts to obtain the inequalities (\ref{ineqmonder}).
\end{proof}

The reader can verify easily that the following two nondifferentiable
functions,
\[
\min\left\{  x_{1},x_{2}\right\}  \text{ and }\max\{x_{1}+x_{2}%
-1,0\},\text{\quad}x_{1},x_{2}\in\lbrack0,1],
\]
are 2-box monotone. They are known in the theory of copulas as the
\emph{Fr\'{e}chet-Hoeffding bounds}. See \cite{Ne2006}. The second function is
convex but the first one is not convex. This fact combined with with Remark
\ref{rem4} shows that the notions of convexity and 2-box monotonicity are
independent in dimension $N\geq2.$

The analogue of $2$-box monotonicity for functions $f$ defined on an interval
$[a,b]$ is the property of \emph{equal increasing increments} (which is
equivalent to convexity):
\[
f(x+z)-f(x)\leq f(y+z)-f(y)
\]
whenever $x\leq y,$ $z>0$ and $x,y,y+z\in\lbrack a,b].$ See \cite{NP2018},
Remark 1.4.1, p. 25 and Corollary 1.4.6, p. 29. An immediate consequence of
this property is the fact that every function of the form
\[
\Phi(\mathbf{x})=f\left(  \langle\mathbf{x},\mathbf{v}\rangle\right)
,\quad\mathbf{x}\in\mathbb{R}^{N},
\]
associated to a convex function $f:\mathbb{R\rightarrow R}$ and a vector
$\mathbf{w}\in\mathbb{R}_{+}^{N}$ is $2$-box convex.

The increment of the log-sum-exp function over the box%
\[
\lbrack0,1]\times\lbrack0,1]\times\{0\}\times\cdots\times\{0\}
\]
equals
\[
\log N-2\log(e+N-1)+\log\left(  2e+N-2\right)  <0
\]
so that, this function is not $2$-box monotone. According to Lemma
\ref{lem2box}, the differential of the log-sum-exp function is not isotone.

The usefulness of the concept of 2-box monotonicity is made clear by the
following approximation result.

\begin{theorem}
\label{thmappr}Suppose that $\Phi$ is a $2$-box monotone convex function
defined on an open box $C$ included in $\mathbb{R}^{N}.$ Then on every compact
box $K\subset C,$ $\Phi$ is the uniform limit of a sequence of infinitely
differentiable strongly convex functions with isotone differentials.

When $C\subset\mathbb{R}_{++}^{N}$ and the function $\Phi$ is also isotone,
then the approximants $\Phi_{n}$ can be chosen to be isotone.
\end{theorem}

\begin{proof}
We use the usual convolution based smooth approximation. Let $\varepsilon>0$
be arbitrarily fixed. Then the function $\Psi=\Phi+\varepsilon\left\Vert
\cdot\right\Vert ^{2}$ is $2$-box monotone and $\varepsilon$-strongly convex.
Besides,
\[
\left\Vert \Psi(\mathbf{x})-\Phi(\mathbf{x})\right\Vert \leq\varepsilon
\sup\left\{  \left\Vert \mathbf{x}\right\Vert ^{2}:\mathbf{x}\in K\right\}
\text{\quad for all }\mathbf{x}\in K.
\]
According to the method of mollifiers, the convolution
\[
(\Psi\ast\varphi)(\mathbf{x})=\int_{\mathbb{R}^{N}}\Psi(\mathbf{x}%
-\mathbf{y})\varphi(\mathbf{y})\mathrm{d}\mathbf{y,}%
\]
of $\Psi$ with any infinitely differentiable function $\varphi:\mathbb{R}%
^{N}\rightarrow\lbrack0,\infty)$ such that $\varphi=0$ on $\mathbb{R}%
^{N}\backslash K$ and $\int_{\mathbb{R}^{N}}\varphi(\mathbf{y})\mathrm{d}%
\mathbf{y}=1,$ is an infinitely differentiable function that provides a
regularization of $\Psi$ since $\Psi\ast\varphi\rightarrow\Psi$ uniformly on
$K$ as the support of $\varphi$ shrinks to $\left\{  0\right\}  .$ An easy
computation shows that $\Psi\ast\varphi$ is also a $2$-box monotone and
$\varepsilon$-strongly convex function. Indeed, with the notation in
Definition \ref{def2boxmon}, we have
\[
\Delta(\Psi\ast\varphi;B_{ij})=\int_{\mathbb{R}^{N}}\Delta(\Psi(\mathbf{x}%
-\mathbf{y});B_{ij})\varphi(\mathbf{y})\mathrm{d}\mathbf{y\geq0}%
\]
and%
\begin{multline*}
(\Psi\ast\varphi)((1-\lambda)\mathbf{u+\lambda v})=\int_{\mathbb{R}^{N}}%
\Psi((1-\lambda)\left(  \mathbf{u-\mathbf{y}}\right)  \mathbf{+\lambda}\left(
\mathbf{v}-\mathbf{y}\right)  )\varphi(\mathbf{y})\mathrm{d}\mathbf{y}\\
\leq\int_{\mathbb{R}^{N}}\left[  \left(  1-\lambda\right)  \Psi\left(
\mathbf{u-\mathbf{y}}\right)  \mathbf{+}\lambda\Psi\left(  \mathbf{v}%
-\mathbf{y}\right)  )-\lambda(1-\lambda)\left\Vert \mathbf{u}-\mathbf{v}%
\right\Vert ^{2}\right]  \varphi(\mathbf{y})\mathrm{d}\mathbf{y}\\
=\left(  1-\lambda\right)  \left(  \Psi\ast\varphi\right)  (\mathbf{u}%
)+\lambda\left(  \Psi\ast\varphi\right)  (\mathbf{v})-\varepsilon
\lambda(1-\lambda)\left\Vert \mathbf{u}-\mathbf{v}\right\Vert ^{2}.
\end{multline*}

Then the conclusion of Theorem \ref{thmappr} follows from Lemma \ref{lem2box}.
\end{proof}

\section{The majorization inequality in the context of ordered Banach spaces}

We start with the case of differentiable convex functions.

\begin{theorem}
\label{thmHLPgen}Suppose that $\sum_{k=1}^{N}\lambda_{k}\delta_{\mathbf{x}%
_{k}}$ and $\sum_{k=1}^{N}\lambda_{k}\delta_{\mathbf{y}_{k}}$ are two discrete
probability measures whose supports are included in an open convex subset $C$
of the ordered Banach space $E$. If $\sum_{k=1}^{N}\lambda_{k}\delta
_{\mathbf{x}_{k}}\prec_{L^{\downarrow}}\sum_{k=1}^{N}\lambda_{k}%
\delta_{\mathbf{y}_{k}},$ then%
\begin{equation}
\sum_{k=1}^{N}\lambda_{k}\Phi(\mathbf{y}_{k})\geq\sum_{k=1}^{N}\lambda_{k}%
\Phi(\mathbf{x}_{k})+\sum_{k=1}^{N}\lambda_{k}\omega(\left\Vert \mathbf{x}%
_{k}-\mathbf{y}_{k}\right\Vert ) \label{Cons1}%
\end{equation}
for every G\^{a}teaux differentiable $\omega$-convex function $\Phi
:C\rightarrow F$ whose differential is isotone, while if $\sum_{k=1}%
^{N}\lambda_{k}\delta_{\mathbf{x}_{k}}\prec_{R^{\uparrow}}\sum_{k=1}%
^{N}\lambda_{k}\delta_{\mathbf{y}_{k}},$ the inequality $(\ref{Cons1})$ works
in the reversed sense.

If $\sum_{k=1}^{N}\lambda_{k}\delta_{\mathbf{x}_{k}}\prec_{wL^{\downarrow}%
}\sum_{k=1}^{N}\lambda_{k}\delta_{\mathbf{y}_{k}},$ then%
\begin{equation}
\sum_{k=1}^{n}\lambda_{k}\Phi(\mathbf{y}_{k})\geq\sum_{k=1}^{n}\lambda_{k}%
\Phi(\mathbf{x}_{k})+\sum_{k=1}^{n}\lambda_{k}\omega(\left\Vert \mathbf{x}%
_{k}-\mathbf{y}_{k}\right\Vert )\text{\quad for }n\in\{1,...,N\} \label{Cons2}%
\end{equation}
whenever $\Phi:C\rightarrow F$ is an isotone and G\^{a}teaux differentiable
$\omega$-convex function whose differential is isotone. Under the same
hypotheses on $\Phi,$ the inequality $(\ref{Cons2})$ works in the reverse way
when $\sum_{k=1}^{N}\lambda_{k}\delta_{\mathbf{x}_{k}}\prec_{wR^{\uparrow}%
}\sum_{k=1}^{N}\lambda_{k}\delta_{\mathbf{y}_{k}}.$
\end{theorem}

\begin{proof}
According to the gradient inequality (\ref{grad_ineq}),%
\begin{multline*}
\sum_{k=1}^{N}\lambda_{k}\Phi(\mathbf{y}_{k})-\sum_{k=1}^{N}\lambda_{k}%
\Phi(\mathbf{x}_{k})=\sum_{k=1}^{N}\lambda_{k}\left(  \Phi(\mathbf{y}%
_{k})-\Phi(\mathbf{x}_{k})\right) \\
\geq\sum_{k=1}^{N}\Phi^{\prime}(\mathbf{x}_{k})(\lambda_{k}\mathbf{y}%
_{k}-\lambda_{k}\mathbf{x}_{k})+\sum_{k=1}^{N}\lambda_{k}\omega\left(
\left\Vert \mathbf{x}_{k}-\mathbf{y}_{k}\right\Vert \right)  ,
\end{multline*}
whence, by using Abel's trick of interchanging the order of summation
(\cite{NP2018}, Theorem 1.9.5, p. 57), one obtains%
\begin{multline*}
D=\sum_{k=1}^{N}\lambda_{k}\Phi(\mathbf{y}_{k})-\sum_{k=1}^{N}\lambda_{k}%
\Phi(\mathbf{x}_{k})-\sum_{k=1}^{N}\lambda_{k}\omega\left(  \left\Vert
\mathbf{x}_{k}-\mathbf{y}_{k}\right\Vert \right) \\
\geq\Phi^{\prime}(\mathbf{x}_{1})(\lambda_{1}\mathbf{y}_{1}-\lambda
_{1}\mathbf{x}_{1})+\sum_{m=2}^{N}\Phi^{\prime}(\mathbf{x}_{m})\Bigl[\sum
_{k=1}^{m}(\lambda_{k}\mathbf{y}_{k}-\lambda_{k}\mathbf{x}_{k})-\sum
_{k=1}^{m-1}(\lambda_{k}\mathbf{y}_{k}-\lambda_{k}\mathbf{x}_{k})\Bigr]\\
=\sum_{m=1}^{N-1}\Bigl[(\Phi^{\prime}(\mathbf{x}_{m})-\Phi^{\prime}%
(\mathbf{x}_{m+1}))\sum_{k=1}^{m}(\lambda_{k}\mathbf{y}_{k}-\lambda
_{k}\mathbf{x}_{k})\Bigr]+\Phi^{\prime}(\mathbf{x}_{N})\left(  \sum_{k=1}%
^{N}(\lambda_{k}\mathbf{y}_{k}-\lambda_{k}\mathbf{x}_{k})\right)  .
\end{multline*}
When $\sum_{k=1}^{N}\lambda_{k}\delta_{\mathbf{x}_{k}}\prec_{L^{\downarrow}%
}\sum_{k=1}^{N}\lambda_{k}\delta_{\mathbf{y}_{k}},$ the last term vanishes and
the fact that $D\geq0$ is a consequence of the isotonicity of $\Phi^{\prime}.$
When $\sum_{k=1}^{N}\lambda_{k}\delta_{\mathbf{x}_{k}}\prec_{wL^{\downarrow}%
}\sum_{k=1}^{N}\lambda_{k}\delta_{\mathbf{y}_{k}}$ and $\Phi$ is isotone, one
applies Lemma \ref{lemAmann_conv} $(a)$ to infer that%
\[
\Phi^{\prime}(\mathbf{x}_{N})\left(  \sum_{k=1}^{N}(\lambda_{k}\mathbf{y}%
_{k}-\lambda_{k}\mathbf{x}_{k})\right)  \geq0.
\]

The other cases can be treated in a similar way.
\end{proof}

The specific statement of Theorem \ref{thmHLPgen} for the class of strongly
convex functions, the class of semiconvex functions as well as its translation
in the case of strongly concave functions and of semiconcave functions is left
to the reader as an exercise. We will detail here only the case of $\sigma
$-smooth functions, which in the light of Lemma \ref{lem4} appears as a
Legendre-Fenchel dual of the majorization inequality.

\begin{theorem}
\label{thmdualmaj}Suppose that $\sum_{k=1}^{N}\lambda_{k}\delta_{\mathbf{x}%
_{k}}$ and $\sum_{k=1}^{N}\lambda_{k}\delta_{\mathbf{y}_{k}}$ are two discrete
probability measures whose supports are included in an open convex subset $C$
of the ordered Banach space $E$. If $\sum_{k=1}^{N}\lambda_{k}\delta
_{\mathbf{x}_{k}}\prec_{L^{\downarrow}}\sum_{k=1}^{N}\lambda_{k}%
\delta_{\mathbf{y}_{k}},$ then%
\begin{equation}
\sum_{k=1}^{N}\lambda_{k}\Phi(\mathbf{y}_{k})\leq\sum_{k=1}^{N}\lambda_{k}%
\Phi(\mathbf{x}_{k})+\frac{\sigma}{2}\sum_{k=1}^{N}\lambda_{k}\left\Vert
\mathbf{x}_{k}-\mathbf{y}_{k}\right\Vert ^{2} \label{maj1sm}%
\end{equation}
for every G\^{a}teaux differentiable and $\sigma$-smooth function
$\Phi:C\rightarrow F$ whose differential is antitone on $C.$

If $\sum_{k=1}^{N}\lambda_{k}\delta_{\mathbf{x}_{k}}\prec_{R^{\uparrow}}%
\sum_{k=1}^{N}\lambda_{k}\delta_{\mathbf{y}_{k}},$ then the conclusion
$(\ref{maj1sm})$ should be replaced by%
\begin{equation}
\sum_{k=1}^{N}\lambda_{k}\Phi(\mathbf{x}_{k})+\sum_{k=1}^{N}\lambda_{k}%
\omega(\left\Vert \mathbf{x}_{k}-\mathbf{y}_{k}\right\Vert )\geq\sum_{k=1}%
^{N}\lambda_{k}\Phi(\mathbf{y}_{k}). \label{maj2sm}%
\end{equation}

Moreover, if the majorization relations $\prec_{L^{\downarrow}}$and
$\prec_{R^{\uparrow}}$are replaced respectively by $\prec_{wL^{\downarrow}}%
$and $\prec_{wR^{\uparrow}},$ then the inequalities $(\ref{maj1sm})$ and
$(\ref{maj2sm})$ still work for those G\^{a}teaux differentiable and $\sigma
$-smooth functions $\Phi:C\rightarrow F$ which are antitone and have antitone differentials.
\end{theorem}

One might wonder if the majorization relations $\prec_{L^{\downarrow}}$and
$\prec_{R^{\uparrow}}$ can be reformulated in terms of doubly stochastic
matrices (as, for example, $\mathbf{x}\prec_{L^{\downarrow}}\mathbf{y}$ if and
only if $\mathbf{x}=P\mathbf{y}$ for some doubly stochastic matrix $P$). The
answer is negative as shows the case of two pairs of elements%
\[
\mathbf{x}_{1}\geq\mathbf{x}_{2}\text{ and }\mathbf{y}_{1}\geq\mathbf{y}%
_{2}=0
\]
such that
\[
\mathbf{x}_{1}\leq\mathbf{y}_{1}\text{ and }\left(  \mathbf{x}_{1}%
+\mathbf{x}_{2}\right)  /2=\left(  \mathbf{y}_{1}+\mathbf{y}_{2}\right)  /2.
\]
Clearly, no $2\times2$-dimensional real matrix $A$ could exist such that
\[
\left(
\begin{array}
[c]{c}%
\mathbf{x}_{1}\\
\mathbf{x}_{2}%
\end{array}
\right)  =A\left(
\begin{array}
[c]{c}%
\mathbf{y}_{1}\\
\mathbf{y}_{2}%
\end{array}
\right)  .
\]
This would imply that $\mathbf{x}_{1},\mathbf{x}_{2}$ belong necessarily to
the segment $[\mathbf{y}_{2}.\mathbf{y}_{1}],$ which is not the case.

However, one implication is true.

\begin{remark}
If $\mathbf{x}_{1}\geq\cdots\geq\mathbf{x}_{N}$ and $\mathbf{y}_{1}\geq
\dots\geq\mathbf{y}_{N}$ are two families of points in the ordered Banach
space $E$ such that%
\[
P\left(
\begin{array}
[c]{c}%
\mathbf{y}_{1}\\
\vdots\\
\mathbf{y}_{N}%
\end{array}
\right)  =\left(
\begin{array}
[c]{c}%
\mathbf{x}_{1}\\
\vdots\\
\mathbf{x}_{N}%
\end{array}
\right)
\]
for a suitable doubly stochastic matrix $P$, then%
\[
\frac{1}{N}\sum_{k=1}^{N}\delta_{\mathbf{x}_{k}}\prec_{L^{\downarrow}}\frac
{1}{N}\sum_{k=1}^{N}\delta_{\mathbf{y}_{k}}.
\]
Indeed, the argument used by Ostrowski $($see \emph{\cite{MOA2011}}, Theorem
$A.4$, p. $31)$ to settle the case $E=\mathbb{R}$ extends verbatim to the case
of ordered Banach spaces.
\end{remark}

In the context of functions of several variables, one can take advantage of
the approximation Theorem \ref{thmappr} to prove the following variant of
Theorem \ref{thmHLPgen}, where the assumption on differentiability is discarded.

\begin{theorem}
\label{thmHLPgen2}Suppose that $C$ is an open box included in $\mathbb{R}%
_{++}^{N}$ and $\sum_{k=1}^{N}\lambda_{k}\delta_{\mathbf{x}_{k}}$ and
$\sum_{k=1}^{N}\lambda_{k}\delta_{\mathbf{y}_{k}}$ are two discrete
probability measures supported at points in $C.$

If $\sum_{k=1}^{N}\lambda_{k}\delta_{\mathbf{x}_{k}}\prec_{L^{\downarrow}}%
\sum_{k=1}^{N}\lambda_{k}\delta_{\mathbf{y}_{k}},$ then%
\[
\sum\nolimits_{k=1}^{n}\lambda_{k}\Phi(\mathbf{y}_{k})\geq\sum\nolimits_{k=1}%
^{n}\lambda_{k}\Phi(\mathbf{x}_{k})
\]
for every $2$-box monotone convex function $\Phi:C\rightarrow F,$ while if

$\sum_{k=1}^{N}\lambda_{k}\delta_{\mathbf{x}_{k}}\prec_{R^{\uparrow}}%
\sum_{k=1}^{N}\lambda_{k}\delta_{\mathbf{y}_{k}},$ the latter inequality works
in the opposite direction.
\end{theorem}

\begin{proof}
Suppose that $\sum_{k=1}^{N}\lambda_{k}\delta_{\mathbf{x}_{k}}\prec
_{L^{\downarrow}}\sum_{k=1}^{N}\lambda_{k}\delta_{\mathbf{y}_{k}}$ and choose
a compact box $K\subset C$ that contains all points $\mathbf{x}_{k}$ and
$\mathbf{y}_{k}.$ According to Theorem \ref{thmappr}, for $\varepsilon>0$
arbitrarily fixed, there is an infinitely differentiable, convex and isotone
function $\Psi_{\varepsilon}$ with isotone differential, such that
$\sup_{\mathbf{x}\in K}\left\vert \Phi(\mathbf{x})-\Psi_{\varepsilon
}(\mathbf{x})\right\vert <\varepsilon.$ Taking into account Theorem
\ref{thmHLPgen}, we infer that%
\[
\sum_{k=1}^{N}\lambda_{k}\Psi_{\varepsilon}(\mathbf{y}_{k})\geq\sum_{k=1}%
^{N}\lambda_{k}\Psi_{\varepsilon}(\mathbf{x}_{k}).
\]
Then%
\[
\sum_{k=1}^{N}\lambda_{k}\Phi(\mathbf{y}_{k})\geq\sum_{k=1}^{N}\lambda_{k}%
\Phi(\mathbf{x}_{k})-2\varepsilon.
\]
As $\varepsilon>0$ was arbitrarily fixed, we conclude that
\[
\sum_{k=1}^{N}\lambda_{k}\Phi(\mathbf{y}_{k})\geq\sum_{k=1}^{N}\lambda_{k}%
\Phi(\mathbf{x}_{k}).
\]
The case when $\sum_{k=1}^{N}\lambda_{k}\delta_{\mathbf{x}_{k}}\prec
_{R^{\uparrow}}\sum_{k=1}^{N}\lambda_{k}\delta_{\mathbf{y}_{k}}$ can be
treated similarly.
\end{proof}

\begin{remark}
\label{remsuf}$($The isotonicity of the differential is not only sufficient
but also necessary for the validity of Theorem \emph{\ref{thmHLPgen} }and
Theorem \emph{\ref{thmHLPgen2}}$)$ As was already noticed in Remark
\emph{\ref{rem4}}, the infinitely differentiable function%
\[
\Phi(x,y)=-2\left(  xy\right)  ^{1/2},\quad x,y\in(0,\infty),
\]
is convex and its differential%
\[
d\Phi:\mathbb{R}_{++}^{2}\rightarrow\mathbb{R}^{2},\text{\quad}d\Phi
(x,y)=-(x^{-1/2}y^{1/2},x^{1/2}y^{-1/2})
\]
is not isotone. Therefore $($see Lemma\emph{ \ref{lem2box}}$)$, the function
$\Phi$ is not $2$-box monotone. Consider the points $\mathbf{x}_{1}%
=(3/2,1)>\mathbf{x}_{2}=(1/2,1)$ and $\mathbf{y}_{1}=(2,2)>\mathbf{y}%
_{2}=(0,0)$. Then%
\[
\mathbf{x}_{1}\leq\mathbf{y}_{1}\text{ and }\left(  \mathbf{x}_{1}%
+\mathbf{x}_{2}\right)  /2=\left(  \mathbf{y}_{1}+\mathbf{y}_{2}\right)  /2,
\]
but%
\[
\Phi(\mathbf{x}_{1})+\Phi(\mathbf{x}_{2})=-2\left(  3/2\right)  ^{1/2}%
-2\left(  1/2\right)  ^{1/2}\approx-3.\,\allowbreak863\,7>\Phi(\mathbf{y}%
_{1})+\Phi(\mathbf{y}_{2})=\allowbreak-4.
\]
Therefore $\Phi$ fails the conclusion of Theorem \emph{\ref{thmHLPgen}}.
\end{remark}

\begin{remark}
In the variant of weak majorization, the assertions of Theorem
\emph{\ref{thmHLPgen2}} remain valid for the $2$-box monotone, isotone and
convex functions defined on an open box included in $\mathbb{R}^{N}.$ Indeed,
in this case the approximants $\Psi_{\varepsilon}$ $($that appear in the proof
of Theorem \emph{\ref{thmappr}}$)$ are not only $2$-box monotone and strictly
convex but also isotone.
\end{remark}

\section{Applications}

The following consequence of Theorem \ref{thmHLPgen} shows that the gap in
Jensen's inequality (the difference of the two sides of this inequality)
decreases when the order interval under attention is shrinking.

\begin{theorem}
\label{thmgapJ}$($The contractibility of Jensen's gap$)$ Suppose that $E$ and
$F$ are ordered Banach spaces, $C$ is an open convex subset of $E$ and
$\Phi:C\rightarrow F$ is a differentiable convex function whose differential
is isotone on $C$. Then for every family of points $\mathbf{x}_{1}%
,\mathbf{x}_{2},\mathbf{y}_{1},\mathbf{y}_{2}$ in $C$ and any $\lambda
\in(0,1)$ such that%
\[
\mathbf{y}_{2}\leq\mathbf{x}_{2}\leq(1-\lambda)\mathbf{y}_{1}+\lambda
\mathbf{y}_{2}\leq\mathbf{x}_{1}\leq\mathbf{y}_{1}.
\]
we have%
\begin{multline*}
0\leq(1-\lambda)\Phi(\mathbf{x}_{1})+\lambda\Phi(\mathbf{x}_{2})-\Phi\left(
(1-\lambda)\mathbf{x}_{1}+\lambda\mathbf{x}_{2}\right) \\
\leq(1-\lambda)\Phi(\mathbf{y}_{1})+\lambda\Phi(\mathbf{y}_{2})-\Phi\left(
(1-\lambda)\mathbf{y}_{1}+\lambda\mathbf{y}_{2}\right)  .
\end{multline*}

\end{theorem}

\begin{proof}
Indeed, under the above hypotheses, we have $\mathbf{x}_{1}\geq(1-\lambda
)\mathbf{y}_{1}+\lambda\mathbf{y}_{2}\geq\mathbf{x}_{2}$ and also%
\begin{align*}
\frac{1-\lambda}{2}\mathbf{x}_{1}  &  \leq\frac{1-\lambda}{2}\mathbf{y}_{1}\\
\frac{1-\lambda}{2}\mathbf{x}_{1}+\frac{1}{2}\left(  (1-\lambda)\mathbf{y}%
_{1}+\lambda\mathbf{y}_{2}\right)   &  \leq\frac{1-\lambda}{2}\mathbf{y}%
_{1}+\frac{1}{2}\left(  (1-\lambda)\mathbf{x}_{1}+\lambda\mathbf{x}_{2}\right)
\\
\frac{1-\lambda}{2}\mathbf{x}_{1}+\frac{1}{2}\left(  (1-\lambda)\mathbf{y}%
_{1}+\lambda\mathbf{y}_{2}\right)  +\frac{\lambda}{2}\mathbf{x}_{2}  &
=\frac{1-\lambda}{2}\mathbf{y}_{1}+\left(  (1-\lambda)\mathbf{x}_{1}%
+\lambda\mathbf{x}_{2}\right)  +\lambda\mathbf{y}_{2}.
\end{align*}
Therefore%
\[
\frac{1-\lambda}{2}\delta_{\mathbf{x}_{1}}+\frac{1}{2}\delta_{(1-\lambda
)\mathbf{y}_{1}+\lambda\mathbf{y}_{2}}+\frac{\lambda}{2}\delta_{\mathbf{x}%
_{2}}\prec_{\mathbf{L^{\downarrow}}}\frac{1-\lambda}{2}\delta_{\mathbf{y}_{1}%
}+\frac{1}{2}\delta_{(1-\lambda)\mathbf{x}_{1}+\lambda\mathbf{x}_{2}}%
+\frac{\lambda}{2}\delta_{\mathbf{y}_{2}}%
\]

so that, taking into account Theorem \ref{thmHLPgen}, we infer that%
\begin{multline*}
(1-\lambda)\Phi(\mathbf{x}_{1})+\Phi\left(  (1-\lambda)\mathbf{y}_{1}%
+\lambda\mathbf{y}_{2}\right)  +\lambda\Phi(\mathbf{x}_{2})\\
\leq(1-\lambda)\Phi(\mathbf{y}_{1})+\Phi\left(  (1-\lambda)\mathbf{x}%
_{1}+\lambda\mathbf{x}_{2}\right)  +\lambda\Phi(\mathbf{y}_{2}),
\end{multline*}
an inequality that is equivalent to the conclusion of Theorem \ref{thmgapJ}.
\end{proof}

A particular case of Theorem \ref{thmgapJ} is as follows:

\begin{corollary}
\label{corpar}$($The parallelogram rule$)$ Suppose that $\Phi:C\rightarrow F$
is as in the statement of Theorem \ref{thmgapJ} and $\mathbf{x}_{1}%
,\mathbf{x}_{2},\mathbf{y}_{1}$ and $\mathbf{y}_{2}~$are points in $C$ such
that $\mathbf{y}_{2}\leq\mathbf{x}_{2}\leq\mathbf{x}_{1}\leq\mathbf{y}_{1}$
and $\left(  \mathbf{x}_{1}+\mathbf{x}_{2}\right)  /2=\left(  \mathbf{y}%
_{1}+\mathbf{y}_{2}\right)  /2$ then the following extension of the
\emph{parallelogram law} takes place:
\[
\Phi(\mathbf{x}_{1})+\Phi(\mathbf{x}_{2})\leq\Phi(\mathbf{y}_{1}%
)+\Phi(\mathbf{y}_{2}).
\]

\end{corollary}

\begin{remark}
$($A multiplicative version of the generalized parallelogram law$)$ Suppose
that $A_{1},A_{2},B_{1},B_{2}~$are positively definite matrices from
$\operatorname*{Sym}(N,\mathbb{R})$ such that
\[
B_{2}\leq A_{2}\leq A_{1}\leq B_{1},\text{\quad}A_{1}A_{2}=A_{2}%
A_{1},\text{\quad}B_{1}B_{2}=B_{2}B_{1},\text{ }%
\]
and $\left(  A_{1}A_{2}\right)  ^{1/2}=\left(  B_{1}B_{2}\right)  ^{1/2}.$
Since the logarithm is an operator monotone function $($see \emph{\cite{Hiai}%
}$)$, we have
\[
\log B_{2}\leq\log A_{2}\leq\log A_{1}\leq\log B_{1}\text{ and }\log
A_{1}+\log A_{2}=\log B_{1}+\log B_{2}.
\]
From Example \emph{\ref{ex_trace} (}presented in Appendix\emph{)} and
Corollary \emph{\ref{corpar} }$($applied to\emph{ }$\operatorname*{trace}%
f(\exp(A))$ we infer that%
\[
\operatorname*{trace}f(A_{1})+\operatorname*{trace}f(A_{2})\leq
\operatorname*{trace}f(B_{1})+\operatorname*{trace}f(B_{2}),
\]
whenever $f:(0,\infty)\mathbb{\rightarrow R}$ is a continuously differentiable
and nondecreasing function such that $f\circ\exp$ is convex.
\end{remark}

\begin{remark}
$($Another variant of the generalized parallelogram law$)$ Suppose that $E$
and $F$ are ordered Banach spaces, $C$ is an open convex subset of $E$ and
$\Phi:C\rightarrow F$ is a differentiable, isotone and convex function whose
differential is isotone on $C$. Then for every family of points $\mathbf{x}%
_{1},\mathbf{x}_{2},\mathbf{y}_{1},\mathbf{y}_{2}$ in $E_{+}$ such that
\[
\mathbf{x}_{2}\leq\mathbf{x}_{1}\leq\mathbf{y}_{1}\text{ and }\mathbf{x}%
_{1}+\mathbf{x}_{2}\leq\mathbf{y}_{1}+\mathbf{y}_{2}%
\]
we have
\[
\Phi(\mathbf{x}_{1})+\Phi(\mathbf{x}_{2})\leq\Phi(\mathbf{y}_{1}%
)+\Phi(\mathbf{y}_{2}).
\]
Indeed, in this case $\mathbf{x}_{1}\leq\mathbf{y}_{1}$ and $\mathbf{x}%
_{1}+\mathbf{x}_{2}\leq\mathbf{y}_{1}+\mathbf{y}_{2}.$ Though $x_{1}%
+x_{2}=y_{1}+y_{2}$ could fail, Theorem \emph{\ref{thmHLPgen}} still applies
because $\Phi^{\prime}(\mathbf{x}_{2})\geq0$ $($see Lemma
\emph{\ref{lemAmann_conv}}$)$.
\end{remark}

Numerous classical inequalities from real analysis can be extended to the
context of ordered Banach spaces via Theorems 4-6. Here are three examples
based on Theorem 4.

\begin{theorem}
\label{thm_S-B}$($The extension of Szeg\"{o} and Bellman inequalities$)$
Suppose that $E$ and $F$ are two ordered Banach spaces, $C$ is an open convex
subset of $E$ that contains the origin and $\Phi:C\rightarrow F$ is a
G\^{a}teaux differentiable $\omega$-convex function whose differential is
isotone. Then for every finite family $\mathbf{x}_{1}\geq\mathbf{x}_{2}%
\geq\cdots\geq\mathbf{x}_{n}\geq0$ of points in $C$ we have%
\begin{multline*}
(1-\sum\nolimits_{k=1}^{n}\left(  -1\right)  ^{k+1})\Phi(0)+\sum
\nolimits_{k=1}^{n}\left(  -1\right)  ^{k+1}\Phi(\mathbf{x}_{k})\geq\Phi
(\sum\nolimits_{k=1}^{n}\left(  -1\right)  ^{k+1}\mathbf{x}_{k})\\
+\sum\nolimits_{k=1}^{n}\omega\left(  \left\Vert x_{k}-x_{k+1}\right\Vert
\right)  +\omega(||\sum\nolimits_{k=1}^{n}\left(  -1\right)  ^{k+1}%
\mathbf{x}_{k}||).
\end{multline*}

\end{theorem}

The proof is immediate, by considering separately the cases where $n$ is odd
or even. The weighted case of Theorem \ref{thm_S-B} can be easily deduced from
it following the argument of Olkin \cite{Olkin} for the strings of real numbers.

\begin{theorem}
Let $f:\mathbb{R}\rightarrow\mathbb{R}$ be a nondecreasing, differentiable and
convex function. If $A_{1},A_{2},...,A_{n}$ and $B_{1},B_{2},...,B_{n}$ are
two families of elements in $\operatorname*{Sym}(N,\mathbb{R})$ such that
\[
A_{1}\geq A_{2}\geq\cdots\geq A_{n}\geq0\text{ and }\sum\nolimits_{k=1}%
^{j}A_{k}\leq\sum\nolimits_{k=1}^{j}B_{k}\text{ for }j\in\{1,2,...,n\},
\]
then
\[
\sum\nolimits_{k=1}^{n}\operatorname*{Trace}f\left(  A_{k}\right)  \leq
\sum\nolimits_{k=1}^{n}\operatorname*{Trace}f\left(  B_{k}\right)  .
\]

\end{theorem}

This is a consequence of Theorem 4 when combined with Example \ref{ex_trace}
in the Appendix. The particular case where $f(x)=x^{2}$ is attributed by Petz
\cite{P} to K. L. Chung.

The third example concerns the case of Popoviciu's inequality. In its simplest
form this inequality asserts that every convex function $\Phi$ defined on a
real interval $I$ verifies the inequality%
\begin{multline*}
\frac{\Phi(x)+\Phi(y)+\Phi(z)}{3}-\Phi\left(  \frac{x+y+z}{3}\right) \\
\geq2\,\left[  \frac{\Phi\left(  \frac{x+y}{2}\right)  +\Phi\left(  \frac
{y+z}{2}\right)  +\Phi\left(  \frac{z+x}{2}\right)  }{3}-\Phi\left(
\frac{x+y+z}{3}\right)  \right]  ,
\end{multline*}
whenever $x,y,z\in I$ (which is an illustration of the contractibility of
Jensen gap in the case of triplets of elements. See \cite{Pop65} and
\cite{NP2018} for details. While Popoviciu's inequality makes sense in any
Banach space, it was shown in \cite{BNP2010} that it actually works only for a
special class of convex functions (including the norm of a Hilbert space).
Based on Theorem \ref{thmHLPgen}, we will show that the class of useful
functions can be enlarged at the cost of limiting the triplets of elements
under consideration.

\begin{theorem}
\label{thmPop}Suppose that $E$ and $F$ are two ordered Banach spaces, $C$ is
an open convex subset of $E$ and $\mathbf{x}\geq\mathbf{y}\geq\,\mathbf{z}$ is
a triplet of points$~$in $C.$ In addition, $\Phi:C\rightarrow F$ is a
G\^{a}teaux differentiable $\omega$-convex function whose differential is isotone.

$(a)$ If $\mathbf{x}\geq(\mathbf{x}+\mathbf{y}+\mathbf{z})/3\geq$\textbf{$y$%
}$\geq\,\mathbf{z},$ then
\begin{multline*}
\frac{\Phi(\mathbf{x})+\Phi(\mathbf{y})+\Phi(\mathbf{z})}{3}+\Phi\left(
\frac{\mathbf{x}+\mathbf{y}+\mathbf{z}}{3}\right) \\
\geq\frac{2}{3}\,\left[  \Phi\left(  \frac{\mathbf{x}+\mathbf{y}}{2}\right)
+\Phi\left(  \frac{\mathbf{y}+\mathbf{z}}{2}\right)  +\Phi\left(
\frac{\mathbf{z}+\mathbf{x}}{2}\right)  \right]  +\frac{1}{6}\omega\left(
\frac{\left\Vert \mathbf{x}-\mathbf{y}\right\Vert }{2}\right) \\
+\frac{1}{6}\omega\left(  \frac{\left\Vert 2\mathbf{z}-\mathbf{x}%
-\mathbf{y}\right\Vert }{6}\right)  +\frac{1}{3}\omega\left(  \frac{\left\Vert
2\mathbf{y}-\mathbf{x}-\mathbf{z}\right\Vert }{6}\right)  +\frac{1}{3}%
\omega\left(  \frac{\left\Vert \mathbf{z}-\mathbf{y}\right\Vert }{2}\right)  .
\end{multline*}

$(b)$ If $\mathbf{x}\geq\mathbf{y}\geq(\mathbf{x}+\mathbf{y}+\mathbf{z}%
)/3\geq\,\mathbf{z},$ then%
\begin{multline*}
\frac{\Phi(\mathbf{x})+\Phi(\mathbf{y})+\Phi(\mathbf{z})}{3}+\Phi\left(
\frac{\mathbf{x}+\mathbf{y}+\mathbf{z}}{3}\right) \\
\geq\frac{2}{3}\,\left[  \Phi\left(  \frac{\mathbf{x}+\mathbf{y}}{2}\right)
+\Phi\left(  \frac{\mathbf{y}+\mathbf{z}}{2}\right)  +\Phi\left(
\frac{\mathbf{z}+\mathbf{x}}{2}\right)  \right]  +\frac{1}{3}\omega\left(
\frac{\left\Vert \mathbf{x}-\mathbf{y}\right\Vert }{2}\right) \\
+\frac{1}{6}\omega\left(  \frac{\left\Vert 2\mathbf{x}-\mathbf{y}%
-\mathbf{z}\right\Vert }{6}\right)  +\frac{1}{3}\omega\left(  \frac{\left\Vert
2\mathbf{y}-\mathbf{x}-\mathbf{z}\right\Vert }{6}\right)  +\frac{1}{6}%
\omega\left(  \frac{\left\Vert \mathbf{z}-\mathbf{y}\right\Vert }{2}\right)  .
\end{multline*}

\end{theorem}

\begin{proof}
$(a)$ In this case%
\[
(\mathbf{x}+\mathbf{y})/2\geq(\mathbf{x}+\mathbf{z})/2\geq(\mathbf{y}%
+\mathbf{z})/2\text{ and }\frac{\mathbf{x}+\mathbf{z}}{2}\geq\mathbf{y.}%
\]
so the conclusion follows from Theorem \ref{thmHLPgen} applied to the families
of points%
\[
\mathbf{x}_{1}=\mathbf{x}_{2}=(\mathbf{x}+\mathbf{y})/2\geq\mathbf{x}%
_{3}=\mathbf{x}_{4}=(\mathbf{x}+\mathbf{z})/2\geq\mathbf{x}_{5}=\mathbf{x}%
_{6}=(\mathbf{y}+\mathbf{z})/2
\]
and
\[
\mathbf{y}_{1}=\mathbf{x}\geq\mathbf{y}_{2}=\mathbf{y}_{3}=\mathbf{y}%
_{4}=(\mathbf{x}+\mathbf{y}+\mathbf{z})/3\geq\mathbf{y}_{5}=\mathbf{y}%
\geq\mathbf{y}_{6}=\mathbf{z,}%
\]
by noticing that%
\begin{align*}
(\mathbf{x}+\mathbf{y})/2  &  \leq\mathbf{x}\\
(\mathbf{x}+\mathbf{y})/2+(\mathbf{x}+\mathbf{y})/2  &  \leq\mathbf{x}%
+(\mathbf{x}+\mathbf{y}+\mathbf{z})/3\\
(\mathbf{x}+\mathbf{y})/2+(\mathbf{x}+\mathbf{y})/2+(\mathbf{x}+\mathbf{z})/2
&  \leq\mathbf{x}+(\mathbf{x}+\mathbf{y}+\mathbf{z})/3+(\mathbf{x}%
+\mathbf{y}+\mathbf{z})/3
\end{align*}
and%
\begin{multline*}
(\mathbf{x}+\mathbf{y})/2+(\mathbf{x}+\mathbf{y})/2+(\mathbf{x}+\mathbf{z}%
)/2+(\mathbf{x}+\mathbf{z})/2\\
\leq\mathbf{x}+(\mathbf{x}+\mathbf{y}+\mathbf{z})/3+(\mathbf{x}+\mathbf{y}%
+\mathbf{z})/3+(\mathbf{x}+\mathbf{y}+\mathbf{z})/3\\
(\mathbf{x}+\mathbf{y})/2+(\mathbf{x}+\mathbf{y})/2+(\mathbf{x}+\mathbf{z}%
)/2+(\mathbf{x}+\mathbf{z})/2+(\mathbf{y}+\mathbf{z})/2\\
\leq\mathbf{x}+(\mathbf{x}+\mathbf{y}+\mathbf{z})/3+(\mathbf{x}+\mathbf{y}%
+\mathbf{z})/3+(\mathbf{x}+\mathbf{y}+\mathbf{z})/3+\mathbf{y}\\
(\mathbf{x}+\mathbf{y})/2+(\mathbf{x}+\mathbf{y})/2+(\mathbf{x}+\mathbf{z}%
)/2+(\mathbf{x}+\mathbf{z})/2+(\mathbf{y}+\mathbf{z})/2+(\mathbf{y}%
+\mathbf{z})/2\\
=\mathbf{x}+(\mathbf{x}+\mathbf{y}+\mathbf{z})/3+(\mathbf{x}+\mathbf{y}%
+\mathbf{z})/3+(\mathbf{x}+\mathbf{y}+\mathbf{z})/3+\mathbf{y}+\mathbf{z}.
\end{multline*}

$(b)$ The proof is similar, by considering the families%
\begin{align*}
\mathbf{x}_{1}  &  =\mathbf{x}_{2}=(\mathbf{x}+\mathbf{y})/2\geq\mathbf{x}%
_{3}=\mathbf{x}_{4}=(\mathbf{x}+\mathbf{z})/2\geq\mathbf{x}_{5}=\mathbf{x}%
_{6}=(\mathbf{y}+\mathbf{z})/2\\
\mathbf{y}_{1}  &  =\mathbf{x},\quad\mathbf{y}_{2}=\mathbf{y},\quad
\mathbf{y}_{3}=\mathbf{y}_{4}=\mathbf{y}_{5}=(\mathbf{x}+\mathbf{y}%
+\mathbf{z})/3,\quad\mathbf{y}_{6}=\mathbf{z.}%
\end{align*}

\end{proof}

In the case where $E=\mathbb{R}$ and $C$ is an interval of $\mathbb{R},$ we
have $\left[  \mathbf{z},\mathbf{x}\right]  =\left[  \mathbf{z},\mathbf{y}%
\right]  \cup\left[  \mathbf{y},\mathbf{x}\right]  ,$ so $(\mathbf{x}%
+\mathbf{y}+\mathbf{z})/3$ lies automatically in one of the intervals $\left[
\mathbf{z},\mathbf{y}\right]  $ and $\left[  \mathbf{y},\mathbf{x}\right]  .$
This allows us to recover the aforementioned result of Popoviciu.

\section{Appendix: Examples of differentiable functions which are isotone
and/or admit isotone differentials}

\begin{example}
\label{ex1}Let $I$ be one of the intervals $(-\infty,0],$ $[0,\infty)$ or
$(-\infty,\infty)).$ The perspective function associated to a convex function
$f:I\rightarrow\mathbb{R}$ is the convex function%
\[
\tilde{f}:I\times(0,\infty),\text{\quad}\tilde{f}(x,y)=yf(x/y).
\]
See \emph{\cite{NP2018}}, Section $3.5$. Assuming $f$ of class $C^{2},$ then%
\[
\frac{\partial\tilde{f}}{\partial x}=f^{\prime}\left(  \frac{x}{y}\right)
,\quad\frac{\partial\tilde{f}}{\partial y}=-\frac{x}{y}f^{\prime}\left(
\frac{x}{y}\right)  +f\left(  \frac{x}{y}\right)  \text{\quad and }%
\frac{\partial^{2}\tilde{f}}{\partial x\partial y}=-\frac{x}{y^{2}}%
f^{\prime\prime}\left(  \frac{x}{y}\right)  .
\]
As a consequence, if $I=(-\infty,0)$, then $d\tilde{f}$ is isotone; if in
addition $f$ is nonnegative and increasing, then $\tilde{f}$ itself is isotone.
\end{example}

\begin{example}
\label{ex4}Let $p\in(1,\infty)$. The function%
\[
\Phi:L^{p}\left(  \mathbb{R}\right)  \rightarrow\mathbb{R},\text{\quad}%
\Phi(f)=\left\Vert f\right\Vert _{p}^{p}=\int_{\mathbb{R}}\left\vert
f\right\vert ^{p}\mathrm{d}t
\]
is convex and differentiable, its differential being defined by the formula%
\[
d\Phi(f)(h)=p\int_{\mathbb{R}}h\left\vert f\right\vert ^{p-1}%
\operatorname*{sgn}f\mathrm{d}t\text{\quad for all }f,h\in L^{p}\left(
\mathbb{R}\right)  .
\]
See \emph{\cite{NP2018}}, Proposition $3.7.8$, p. $151$. Clearly, $\Phi$ and
its differential are isotone on the positive cone of $L^{p}\left(
\mathbb{R}\right)  .$ A variant of this example within the framework of
Schatten classes is provided by Theorem $16$ in \emph{\cite{KST}}.
\end{example}

\begin{example}
\label{ex5}The negative entropy function, $E\left(  \mathbf{x}\right)
=\sum_{k=1}^{N}x_{k}\log x_{k},$ is $C^{\infty}$-differentiable on
\[
\mathbb{R}_{++}^{N}=\{\mathbf{x}=\left(  x_{1},...,x_{N}\right)  \in
\mathbb{R}^{N}:x_{1},\ldots,x_{N}>0\}
\]
and strongly convex on any compact subset $K$ of $\mathbb{R}_{++}^{N}$. The
differential of $\Phi$ is the map $d\Phi:\mathbb{R}_{++}^{N}\rightarrow\left(
\mathbb{R}^{N}\right)  ^{\ast}$ given by the formula
\[
d\Phi(\mathbf{x})\mathbf{v}=\sum\nolimits_{k=1}^{N}(1+\log x_{k}%
)v_{k},\text{\quad}\mathbf{x}\in\mathbb{R}_{++}^{N}~\text{and }\mathbf{v}%
\in\mathbb{R}^{N},
\]
so that $\mathbf{x}\leq\mathbf{y}$ in $\mathbb{R}_{++}^{N}$ implies
$d\Phi(\mathbf{x})\leq d\Phi(\mathbf{y}).$
\end{example}

\begin{example}
\label{ex6}The log-sum-exp function is defined on $\mathbb{R}^{N}$ by the
formula
\[
\operatorname*{LSE}(\mathbf{x})=\log(\sum\nolimits_{k=1}^{N}e^{x_{k}%
}),\text{\quad}\mathbf{x}\in\mathbb{R}^{N}.
\]
This function is infinitely differentiable, isotone and convex, but it is not
strongly convex. See \emph{\cite{NP2018}}, Example $3.8.9$, pp. \emph{157-158}%
. A simple argument showing that the differential of $\operatorname*{LSE}$ is
not isotone is given in the comments after Lemma \emph{\ref{lem2box}}. The
log-sum-exp function is the Legendre-Fenchel conjugate of the restriction of
the negative entropy function $E$ to the simplex $\Delta=\left\{
x:\mathbf{x}\in\mathbb{R}^{N},\text{ }\sum_{k=1}^{N}x_{k}=1\right\}  .$ See
\emph{\cite{BV2004}}, p. $93$. Since $E$ is strongly convex, it follows from
Lemma \emph{\ref{lem4}} that the log-sum-exp function is strongly smooth.
\end{example}

\begin{example}
\label{ex_trace}$($Trace functions of matrices$)$ Denote by
$\operatorname*{Sym}(N,\mathbb{R)}$ the ordered Banach space of all $N\times
N$-dimensional symmetric matrices with real coefficients endowed with the
Frobenius norm and the \emph{L\"{o}wner ordering},%
\[
A\leq B\text{ if and only if }\langle A\mathbf{x},\mathbf{x}\rangle\leq\langle
B\mathbf{x},\mathbf{x}\rangle\text{ for all }\mathbf{x}\in\mathbb{R}^{N}.
\]
If $f:\mathbb{R\rightarrow R}$ is a continuously differentiable $($strongly$)$
convex function, then the formula
\[
\Phi(A)=\operatorname*{trace}(f(A))
\]
defines a differentiable $($strongly$)$ convex function on
$\operatorname*{Sym}(n,\mathbb{R})$. Since $d\Phi(A)X=\operatorname*{trace}%
\left(  f^{\prime}(A)X\right)  $ and $f^{\prime}$ is isotone, it follows that
$d\Phi$ is isotone too. According to\ Weyl's monotonicity principle $($see
\emph{\cite{NP2018}}, Corollary $4.4.3$, p. $203),$\emph{ }the function $\Phi$
is isotone if $f$ itself is isotone.

Two particular cases are of a special interest:

$(a)$ The operator analogue of the negative entropy function presented in
Example $4$ is the negative von Neumann entropy, defined on the compact convex
set $C=\left\{  A\in\operatorname*{Sym}\nolimits^{++}(N,\mathbb{R}%
):\operatorname*{trace}(A)=1\right\}  $ via the formula
\[
S(A)=\operatorname*{trace}\left(  A\log A\right)  =\sum\nolimits_{k=1}%
^{N}\lambda_{k}(A)\log\lambda_{k}(A),
\]
where $\lambda_{1}(A),...,\lambda_{N}(A)$ are the eigenvalues of $A$ counted
with their multiplicity. According to the preceding discussion, this function
is convex and differentiable and its differential is isotone. One can prove
$($using Lemma \emph{\ref{lem4}}$)$ that the negative von Neumann entropy is
$1/2$-strongly convex and its Legendre-Fenchel conjugate, the convex function
$\log(\operatorname*{trace}(e^{A}))),$ is $2$-smooth. See \emph{\cite{KST}},
Theorem $16).$

$(b)$ The function $\operatorname*{trace}(e^{A})$ is $\log$-convex and
continuously differentiable, with isotone differential. However, the
differential of the convex function $\log(\operatorname*{trace}(e^{A}))$ is
not anymore isotone. See Example $5,$ which discusses the case of diagonal matrices.
\end{example}

\medskip

\noindent\textbf{\noindent Acknowledgement. }The author would like to thank
\c{S}tefan Cobza\c{s}, Sorin G. Gal and Flavia-Corina Mitroi-Symeonidis for
useful conversations on the subject of this paper and to the reviewer for many
valuable and constructive comments that have improved the final version of the paper.

\end{document}